\documentclass[]{interact}

\usepackage{epstopdf}
\usepackage{ulem}
\usepackage[round,longnamesfirst]{natbib}
\bibpunct[, ]{[}{]}{,}{n}{,}{,}
\renewcommand\bibfont{\fontsize{10}{12}\selectfont}
\makeatletter
\usepackage{amsmath,amsthm,color,hyperref}
\usepackage{amssymb}
\usepackage{amsfonts}
\usepackage{graphicx}
\usepackage{latexsym}

\usepackage{subfigure}

\newcommand{\R}{\mathbb{R}}
\newcommand{\E}{\mathbb E}

\def\1{\mathbb{I}}

\numberwithin{equation}{section}
\theoremstyle{plain}

\newcounter{thm}[section]
\newcounter{appen}

\newtheorem{theor}[thm]{Theorem}
\newtheorem{defi}[thm]{Definition}

\newtheorem{lem}[thm]{Lemma}
\newtheorem{remark}[thm]{Remark}
\newtheorem{proposition}[thm]{Proposition}

\newtheorem{lemma}[appen]{Lemma}

\begin{document}
	

	\title{Least square estimators in linear regression models under negatively superadditive dependent random observations}

\author{
\name{Karine Bertin\textsuperscript{a}\thanks{CONTACT Karine Bertin. Email: karine.bertin@uv.cl, Final version will be published in Statistics}, Soledad Torres\textsuperscript{a} and Lauri Viitasaari\textsuperscript{b}}
\affil{\textsuperscript{a}CIMFAV-INGEMAT, Universidad de Valpa{r}a\'{\i}so, Chile; \textsuperscript{b}Uppsala University, Department of Mathematics, Sweden}
}
\date{}
\maketitle
	
	%
	%
	%
	%
	%
	
	
\begin{abstract}
	In this article we study the asymptotic behaviour of the least square estimator in a linear regression model based on random observation instances. We provide mild assumptions on the moments and dependence structure on the randomly spaced observations and the residuals under which the estimator is strongly consistent. In particular, we consider observation instances that are negatively superadditive dependent within each other, while for the residuals we merely assume that they are generated by some continuous function. We complement our findings with a simulation study providing insights on finite sample properties.
		
\end{abstract}


	%
	%
	%
	%
\section{Introduction}

In this paper, we study {the strong consistency of the least square estimator in} a simple regression model with a linear trend and continuous noise where the observation measurements are made at random times $\tau_i$. More precisely,
consider the simple regression model
\begin{equation}\label{modelo}
Y_{\tau_{i}} = a  \tau_{i} + \varepsilon_{\tau_{i}}, \quad  {1\le i\le N(T)}
\end{equation}
whose unknown parameter $a$ must be estimated. Here $\tau := \{ \tau_i \}_{i\in\mathbb{N}}$ is a random increasing sequence of {positive} random variables
and $N(T)=\sum_{j\geq 1} 1_{\tau_j \leq T}$ determines the number of observations in $[0,T]$ with $T>0$.
Random time observation models arise in many domains. In medicine,  Lange et al. \cite{Lange} study  the  disease process  according to a latent continuous-time Markov chain  with observation rates that depend on the individual's underlying disease status. In finance, A\"it-Sahalia and Mykland \cite{sahalia}  consider random times in the case of continuous time diffusions.
In  biomedical studies, the event of interest can occur more than once in a participant. This recurrent events can be thought as correlated random periods.
The effect of the random sampling compared to discrete ones have been considered in  \cite{Mc}.  They study the estimation of the mean and autocorrelation parameter of a stationary Gaussian process under two sampling schemes:  sampling the continuous time process at fixed equally-spaced times and  sampling at random times based on a  renewal process.
In non-parametric statistics, Masry \cite{masry1} studies the problem of estimating an unknown probability density function, based on $n$ independent observations sampled at random times. Vilar and Vilar
(\cite{vilar}, \cite{vilar2000}) investigate non-parametric regression estimation with randomly spaced observations associated to renewal processes.
Strong consistency of the least-square estimator of the trend in model \eqref{modelo} is proved in  \cite{araya} and \cite{araya1}. They consider i.i.d. random times based  either on jittered sampling or renewal processes, and increments of fractional Brownian motion as the noise.

In this paper, we extend the results of Araya et al. (\cite{araya}, \cite{araya1})  by considering more general assumptions on the noise $\varepsilon$ and dependent observation measurements  $\tau_i$ characterized by negative superadditive dependent (NSD) positive random variables. Additionally, only observations until a time $T>0$ are considered, assuming, without loss of generality, $T=1$.

NSD dependence for random variables have been first introduced  by Hu \cite{hu2000negatively} who extends the concept of negatively associated (NA) dependence (see \cite{AS}). NSD allows one to get many important inequalities and convergence results.  Eghbal et al. \cite{E1} derive convergence results of sums of quadratic forms of NSD random variables under the assumption of existence of  moment  of order $r>1$.  Shen et al. \cite{S1} study almost sure convergence and strong stability for weighted sums of NSD random variables.
Shen et al. \cite{S2} investigate strong convergence for NSD random variables and present some moment inequalities.
Wang et al. \cite{Wang} study complete convergence for arrays of rowwise NSD random variables with applications to nonparametric regression.

In this work, we prove the strong consistency of the least-square estimator of the trend $a$ and derive the rate of convergence. For the noise $\varepsilon$, we merely assume that it arises from increments of a continuous function.
We do not pose any requirements on the dependence structure between $\varepsilon$ and the random times $\tau_i$. On top of minimal requirements on $\varepsilon$, this fact provides additional flexibility to our model class.
We illustrate the performances of our estimator in a simulation study where random times are given by sum of NSD log-normal variables.

The paper is organized as follows. Section \ref{sec:2} presents the model, the assumptions on the random times, and our main results. Section~\ref{sim} is devoted to the simulation study and Section~\ref{proofs} to the proofs.

\section{Linear regression model with random time observations}
\label{sec:2}
We begin by introducing the underlying linear regression model and assumptions on the random observations required for the proofs of the main results, Theorem \ref{thm:N1} and Theorem \ref{thm:N}.

We assume that observations are made on random time instances $(\tau_{i})_{i\in \mathbb{N}}$ determined by an increasing sequence of random variables $\tau_i$. In order to force observation instances $\tau_i$ into a fixed time interval $[0,T]$, we denote by $N(T)$ the amount of observations needed to reach the fixed time $T$. That is, we set
 {\begin{equation}\label{eq:N(1)}N(T) = \arg\max_{i \in \mathbb{N} } \{\tau_i :  \,\, \tau_i \leq T\}=\sum_{j\geq 1} 1_{\tau_j \leq T}.
 \end{equation}}
Without loss of generality and for the sake of simplicity, we let $T=1$ and consider $N(1)$ as the number of observation time instances before time $T=1$. The observation times $\tau_i$ are given by
{\begin{equation}\label{tauin}
\tau_{i}=\sum_{j=1}^i t_{j}, i \ge 1,
\end{equation}}
where now $\{t_{j}\}_{j\in\mathbb{N}}$ is a sequence of positive random variables. 
With this notation, we consider
the linear regression model
\begin{equation}\label{reg}
Y_{\tau_{i}} =  a \tau_{i} + \varepsilon_{\tau_{i}}, \quad i=1,\ldots, N(1).
\end{equation}
Here $a\in\R$ is the unknown parameter to be estimated, and the residuals $\varepsilon_{\tau_i}$ are given by
\begin{equation*}
\varepsilon_{\tau_{i}}=W_{\tau_{i}}-W_{\tau_{i-1}},
\end{equation*}
where $W$ is an arbitrary (almost surely) continuous function {and we fix $\tau_0=0$. } We stress that the only assumption on $W$ we require is the continuity. Consequently, $W$ can be chosen to be a path of any almost surely continuous stochastic process or even a deterministic continuous function.

We estimate the model parameter $a$ by
the classical least square estimator (LSE) with random number of observations $N(1)$, given by
\begin{equation*}
\hat{a}_{N(1)}=  \frac{\displaystyle \sum_{i=1}^{N(1)} Y_{\tau_{i}}{\tau_{i}}}{\displaystyle \sum_{i=1}^{N(1)}{\tau_{i}^2}}.
\end{equation*}
Clearly, we have
\begin{equation}
\label{eq:estimator-difference}
\hat{a}_{N(1)}-a=\frac{\displaystyle \sum_{i=1}^{N(1)} {\tau_{i}}\left(W_{\tau_{i}}-W_{\tau_{i-1}}\right)}{\displaystyle \sum_{i=1}^{N(1)} {\tau_{i}^2}}.
\end{equation}
{In order to obtain strong consistency and the rate of convergence for the estimator (cf. Theorem \ref{thm:N1}), we need to pose assumptions on 
time instances $t_i$ that ensure a sufficient amount of observations in the interval $[0,1]$.  We fix $N\in \mathbb{N}$.} As our first assumptions, we pose the following moment conditions.
\begin{itemize}
\item[$(H_1)$] $\mathbb{E}(t_{i})=\frac{1}{N}$ for all $i$ (and $N$).
\item[$(H_2)$] There exists $r>0$ and a constant $\tilde{C}>0$ such that, for all $i$ and $N$, we have $\mathbb{E}(t_{i}^{2+r})\le \frac{\tilde C}{N^{2+r}}$.
\end{itemize}

{Note that the number N describing the means of random variables $t_i$ can be roughly understood as the sampling rate (average number of observations in $[0,1]$). In fact, below in Lemma 4.4 we show that the number of observations $N(1)$ on $[0,1]$ satisfies $N(1)/N \to 1$ as $N$ tends to $\infty$. Note also that the observation times $\tau_i$ as well as the generating random variables $t_i$ depend on $N$. However, for the sake of readability, we omit this dependence in the notation.}


{Moreover, we stress that assumptions $(H_1)$ and $(H_2)$ are rather mild. Condition $(H_1)$ ensures, as we said previously, that $N(1)$ is roughtly proportional to $N$ (see Lemma \ref{lma:N1-N-ratio}). }
Condition $(H_2)$ requires slightly better than square integrability, and thus one can consider even random times with relatively heavy tails.
Moreover, the bound of $(H_2)$ ensures that random times $t_i$ are sufficiently concentrated around their mean $1/N$. This in turn implies that consecutive observation times $\tau_i$ and $\tau_{i+1}$ are never too far from each other.

In addition to the moment conditions $(H_1)$ and $(H_2)$ above, our third hypothesis  is related to the dependence structure within the sequence $ \{  t_i \}_{i\in\mathbb{N}}$. For this we need some preliminary definitions.
\begin{defi}(Kemperman \cite{kemperman1977fkg}). A function $\phi  : \mathbb{R}^m \mapsto  \mathbb{R}$ is called superadditive, if  for all $ x, y \in  \mathbb{R}^m$ we have
$$ \phi (x \vee y) + \phi (x \wedge  y) \ge  \phi(x) + \phi(y),$$ where $\vee $  stands for componentwise maximum and $\wedge$ for componentwise minimum.
\end{defi}
A characterization of smooth superadditive functions is given in the following lemma.
\begin{lemma}\label{lem:kemp} (Kemperman \cite{kemperman1977fkg}).
If $\phi$ has continuous second partial derivatives, then the superadditivity of $\phi$ is equivalent to $\frac{\partial ^2 \phi }{\partial x_i x_j} \ge 0,1 \leq  i \neq j \leq  m$.
\end{lemma}
Negatively superadditive dependent (NSD) variables are defined in the following way.
\begin{defi}\label{def:hu}
(Hu \cite{hu2000negatively}). A random vector $(X_1,X_2, \ldots ,X_m)$  is said to be negatively superadditive dependent, if for any  superadditive function $\phi$ such that the expectation $\mathbb{E} \phi(X_1,X_2, \ldots ,X_m)$ exists, we have
$$\mathbb{E} \phi(X_1,X_2, \ldots ,X_m) \leq \mathbb{E}\phi({X}^*_1,{X}^*_2, \ldots , {X}^*_m),$$
where ${X}^*_1 , {X}^*_2 , \ldots  , {X}^*_m$  are independent and, for all $1\leq i\leq m$,  $X_i$ and ${X}^*_i$ are equally distributed.
\end{defi}
Any vector consisting of independent random variables $X_i$ is trivially NSD. Many non-trivial examples of NSD variables with values in $\R$ or $\R^m$ are given in the literature. In particular, some classes of elliptical NSD variables are given in \cite{block1988conditionally} and \cite{hu2000negatively}. More specifically, variables with a density given by
\begin{equation*}
g(x)=|\Sigma|^{-1/2} g\left( x^t\Sigma x\right), \quad x\in\R^m
\end{equation*}
are NSD, provided that $\Sigma=(\sigma_{i,j})$ satisfies $\sigma_{i,j}<0$ for $i\neq j$ and $\int_0^\infty g(t) t^{m/2-1}dt <\infty$.
As a consequence, a random vector that follows a multivariate normal distribution with covariance matrix $\Sigma=(\sigma_{i,j})$ satisfying $\sigma_{i,j}<0$ for $i\neq j$ is NSD.
Our third and final assumption on the random times $t_i$ is the following.
\begin{itemize}
\item[$(H_3)$] For all $m\in\mathbb{N}$ (and $N$), the vector $(t_1,\ldots,t_m)$ is NSD.
\end{itemize}
\begin{remark}
\label{remark:second-moment}
Note that $(H_1)$ and $(H_3)$ imply that, for any $i\neq j$,
$$\mathbb{E}(t_{i}t_{j})\le \frac{1}{N^2}.$$ Indeed, {since the function $(x_1,x_2)\mapsto x_1x_2$ is superadditive, }
$\mathbb{E}(t_{i}t_{j})\le \mathbb{E}(t^*_{i}t^*_{j})=\frac{1}{N^2}$, where the equality follows from the independence of $t^*_i$ and $t^*_j$. Note also that $(H_2)$ together with H\"older inequality gives us
$$\mathbb{E}(t_{i}^2)\le \frac{\tilde{C}^{2/(2+r)}}{N^2}.$$
\end{remark}
\begin{remark}
Many non-trivial examples of NSD vectors presented in the literature take values in $\R^m$ while for our purposes the random times $t_i$ are assumed to be non-negative.
A simple approach to construct non-negative NSD vectors, suitable for our purposes, is to choose an $\R^m$-valued NSD vector $(X_1,\dots,X_m)$ and a non-decreasing function $\varphi:\R\to\R_+$. Then by a routine exercise one can show that $(\varphi(X_1),\dots,\varphi(X_m))$ is non-negative and NSD. Indeed, with Lemma~\ref{lem:kemp} one can show that if $\phi:\R^m\to\R$ is superadditive and $\varphi:\R\to\R_+$ is non-decreasing, then the function $\tilde{\phi}: \R^m\to\R$ defined by
\begin{equation}
\tilde{\phi}(x_1,\dots,x_m)=\phi(\varphi(x_1),\dots,\varphi(x_m))
\end{equation}
is superadditive. As a particular example, a vector $(Y_1,\dots,Y_m)$ of multivariate log-normal variables $Y_i$ with parameter $\Sigma=(\sigma_{i,j})$ satisfying $\sigma_{i,j}<0$ for $i\neq j$ is NSD. Indeed, in this case we have $Y_i=\exp(X_i)$ for all $1\leq i\leq m$, where $(X_1,\dots,X_m)$ is a multivariate normal NSD vector. This example is used in our simulation study, provided in Section \ref{sim}.
\end{remark}
We are now ready to state the main results of this paper.
\begin{theor}
\label{thm:N1}
Suppose that the sequence of random times  $\{t_{j} \}_{j \in \mathbb{N}}$  satisfies $(H_1)$ to $(H_3)$ and that $W$ is continuous almost surely. Then we have, almost surely as $N\to \infty$,
{$$
N(1)\left(\hat{a}_{N(1)}-a\right) \to 3\int_0^1 (W_1 - W_s) ds,
$$}
where  $N(1)$ is defined in \eqref{eq:N(1)}.
\end{theor}
If for some particular reason the observation window does not have to be restricted to the time interval $[0,1]$, then one can use all observations in the estimation and consider an estimator
\begin{equation*}
\hat{a}_{N}=  \frac{\displaystyle \sum_{i=1}^{N} Y_{\tau_{i}}{\tau_{i}}}{\displaystyle \sum_{i=1}^{N}{\tau_{i}^2}}.
\end{equation*}
In this case the result can be formulated similarly as in Theorem \ref{thm:N1}.
\begin{theor}
\label{thm:N}
Suppose that the sequence of random times  $\{ t_{j} \}_{j \in \mathbb{N}}$ satisfies hypothesis  $(H_1)$ to $(H_3)$ and that $W$ is continuous almost surely. Then we have, almost surely as $N\to \infty$,
{$$
N\left(\hat{a}_{N}-a\right) \to 3\int_0^1 (W_1 - W_s) ds.
$$}
\end{theor}

\begin{remark} {Note that the classical fixed design model is included in our framework since in this case the variables $t_j$ are constant variables equal to $\frac{1}{N}$ and they satisfy hypotheses $(H_1)$, $(H_2)$ and $(H_3)$.}
Another type of random time sampling, called jittered sampling, is studied in \cite{vilar} and in \cite{araya1}, among others. Jittered sampling corresponds to the case where the sampling is nearly regular in the sense that the fluctuations from the mean are small compared to the sampling interval. More precisely, in jittered case the random times $\{ \tau_{j} \}_{j \in \mathbb{N}}$ are given by
$$
\tau_j = \frac{j}{N} + \frac{X_j}{N}, \,\, j=1, \ldots, N-1
$$
where for each $j=1,2,\ldots, N-1$ the random variable $X_j$ is supported on $\left(-\frac12,\frac12\right)$ and has zero expectation. In order to extend our results to cover this case as well, a careful examination of our proofs below reveals that the statements of Theorem \ref{thm:N1} and Theorem \ref{thm:N} are valid, for any continuous $W$, provided that, as $N\to \infty$, we have:
\begin{itemize}
	\item $
	\max_{1\leq k\leq 2N}|\tau_k - \tau_{k-1}| \to 0,
	$
	\item $\tau_N \to 1$ and $\tau_{N(1)}\to 1$,
	\item $\frac{N(1)}{N}\to 1$, and
	\item $\frac{1}{N}\sum_{i=1}^ N \tau_i^2 \to 1/3$.
\end{itemize}
It is straightforward to check that all these conditions are satisfied in the jittered case as well regardless on the dependence structure within the random variables $(X_j)_{j\in \mathbb{N}}$, complementing our work by covering the jittered case as well.
\end{remark}

\begin{remark}
{Note that, by following the techniques used in this paper, our results could be extended to cover several generalisations of our model. For example, one could add an intercept parameter to  \eqref{reg} leading to a model
\begin{equation*}\label{reg-int}
Y_{\tau_{i}} =   \alpha + a \tau_{i} + \varepsilon_{\tau_{i}}, \quad i=1,\ldots, N(1).
\end{equation*}
	Another interesting generalisations would be time series models such as an AR(1) model
	\begin{equation*}\label{AR1}
	Y_{\tau_{i+1}} =  \theta Y_{\tau_{i}}  + \varepsilon_{\tau_{i}}, \quad i=1,\ldots, N(1)
	\end{equation*}
or non-parametric regression model given by
		\begin{equation*}\label{NP}
	Y_{\tau_{i}} =  r(x)  + \varepsilon_{\tau_{i}}, \quad i=1,\ldots, N(1).
	\end{equation*}
In addition, one could easily extend our results to cover multivariate models with suitable correlation structures within the components as well.}
\end{remark}

\section{Simulation}\label{sim}
In this section we illustrate the performance of the estimator  $\hat{a}_{N(1)}$ defined in \eqref{eq:estimator-difference}. For the process $W$ we take the fractional Brownian motion and the variables $t_j$ follow an NSD log-normal distribution.
  \\
  \ \\
 {\bf The fractional Brownian motion as the noise:}
We take $W = B^H$, where $B^H$ is a fractional Brownian motion with Hurst parameter $H\in\{0.1,0.5,0.9\}$.
This process is one of the most popular Gaussian stochastic processes with memory and the case $H=0.5$ corresponds to independent residuals. 
We recall  the main properties of the stochastic  process $B^H$:
\begin{itemize}
	\item $B^H$ is a Gaussian selfsimilar stochastic process with $B^H_0=0$.
	\item The covariance of $B^H$ is given by $R_H (t, s) = \mathbb{E}(B_t^H B_s^H ) = \frac{ \sigma^2}{2} \left(  t^{2H} + s^{2H} - |t-s|^{2H}  \right) $.
\end{itemize}
In this case the variable $3\int_0^1 (W_1 - W_s) ds$ that appears in Theorems~\ref{thm:N1} and~\ref{thm:N} is normally distributed with mean $0$ and variance $\sigma_H=9/(2H+2)$.\\
\ \\
{\bf Distribution of the random times:}
In the following we specify how we obtain $t_1,\ldots,t_{N(1)}$ and $\tau_1,\ldots,\tau_{N(1)}$. {Since $N(1)$ is obtained by \eqref{eq:N(1)}, we have to generate a number $N_0\in\mathbb{N}$ of variables $t_j$ and $\tau_j$ large enough to ensure $N(1)$ observations. In fact, in our simulation study  $N(1)$ will be obtained by
$$N(1)=\sum_{j=1}^{N_0} 1_{\tau_j\le 1}.$$}
We simulate multivariate normal variables $(X_1 , \ldots , X_{N_0})$ with mean $\bf{\mu}=(\mu_1,\ldots,\mu_{N_0})$ and covariance matrix $\Sigma =(\sigma_{i,j})_{i,j=1, \ldots ,  N_0}$. The mean and covariance are given, for $i,j\in\{1,\ldots,N_0\}$, by $\mu_i=-\log(N)-1/2$, $\sigma_{i,i}=1$, and {$\sigma_{i,j}=-1/4, j\neq i$.} For $i=1,\ldots,N_0$ we define the variables $t_i=\exp{X_i}$. These variables are log-normal with mean $\frac{1}{N}$ (see $(H_1)$).
It can be easily checked also that the variables $t_i$ satisfy assumptions $(H_2)$ and $(H_3)$. The random times $\tau_1,\dots,\tau_{N(1)}$ are then obtained by \eqref{eq:N(1)} and \eqref{tauin}. \\
\ \\
Finally, for the model parameter we set $a=1$. Realisations of the data $Y_{\tau}$ for $N=1000$ and $H\in\{0.1,0.5,0.9\}$ are plotted in Figure~\ref{data}. Due to the H\"older regularity of the fractional Brownian motion $B^H$, the data are more noisy for lower values of $H$.
\begin{figure}[htp!]
		\centering
	\begin{subfigure}
		\centering
		\includegraphics[width=3.9cm]{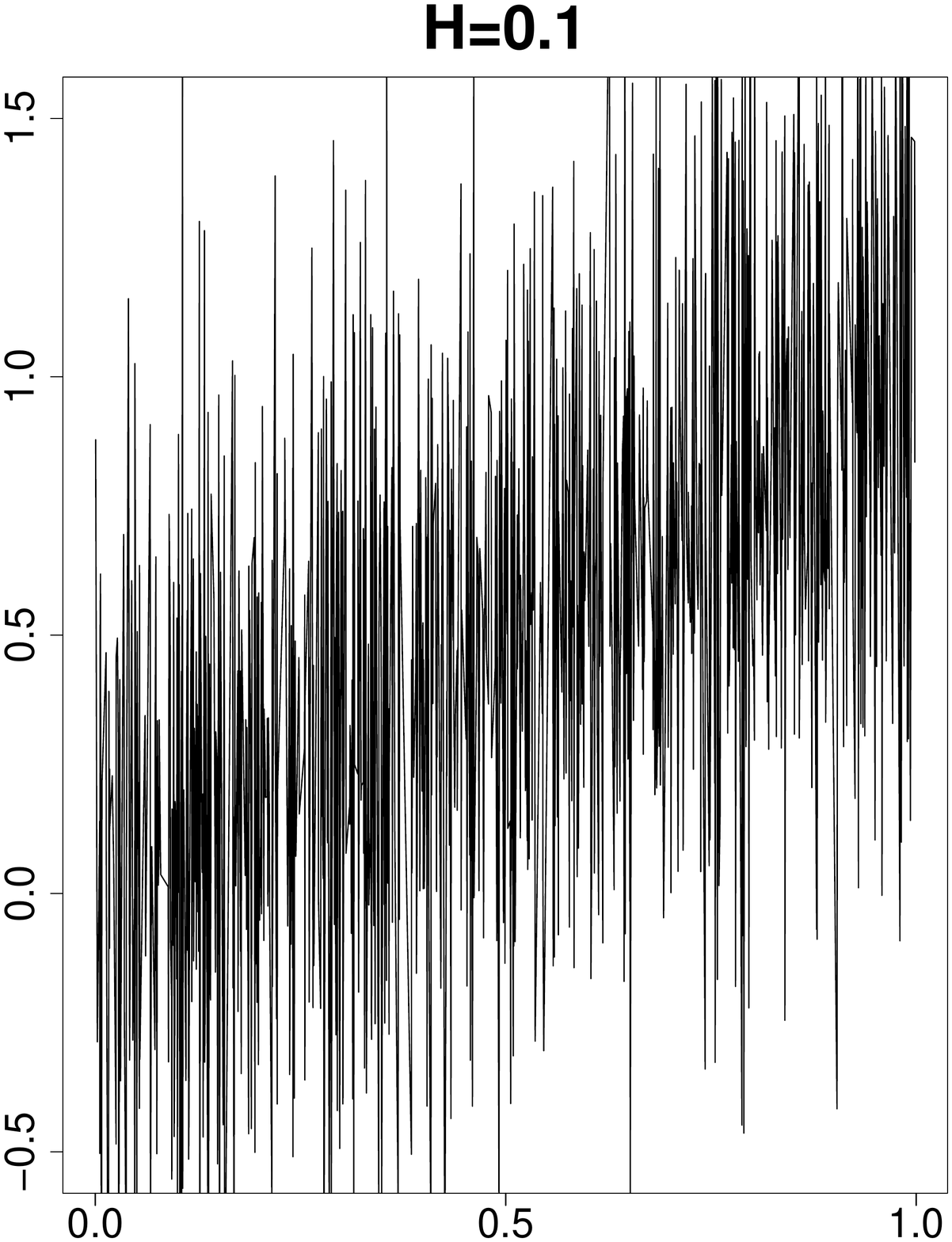}
	\end{subfigure}
	\begin{subfigure}
		\centering
		\includegraphics[width=3.9cm]{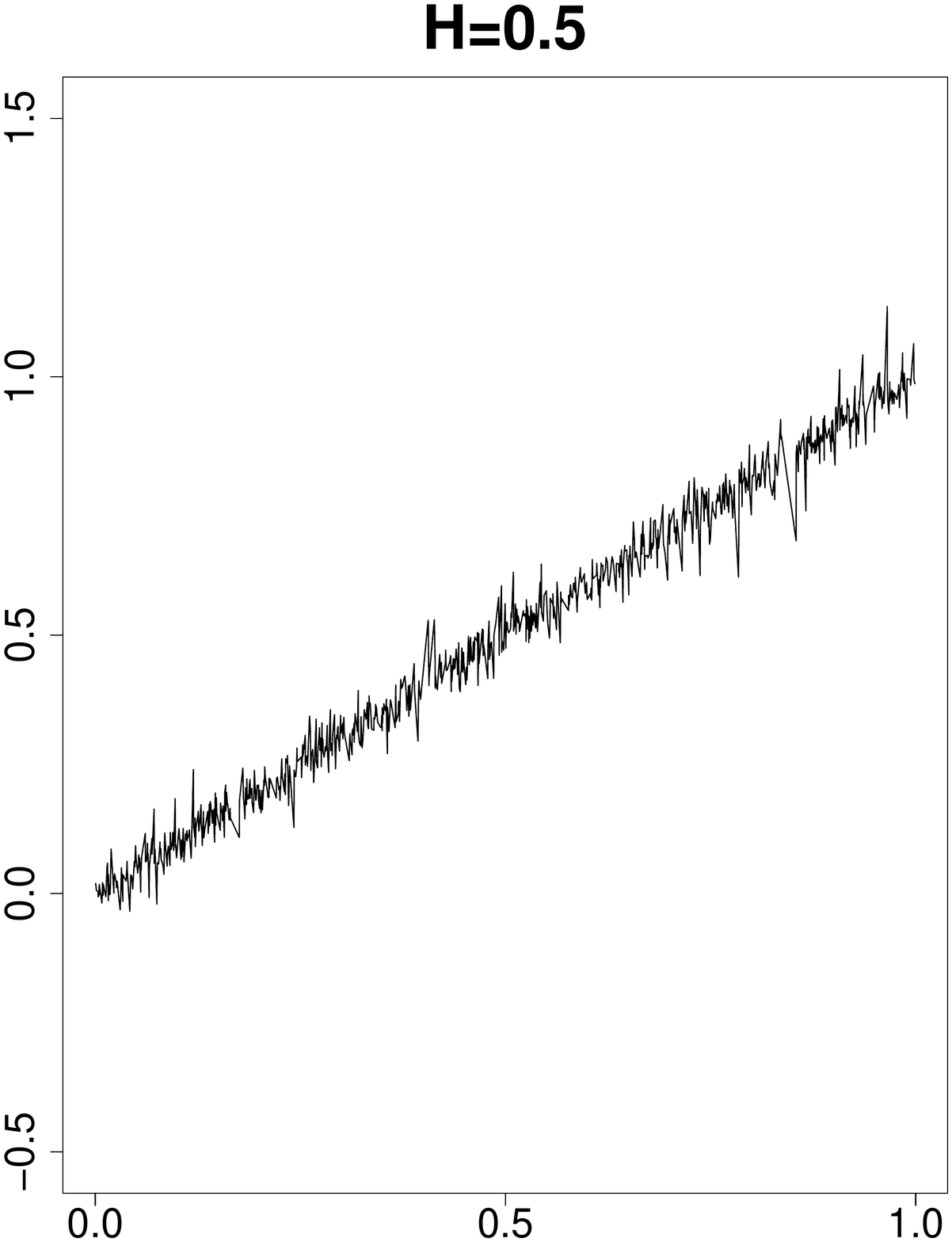}
	\end{subfigure}
	\begin{subfigure}
	\centering
	\includegraphics[width=3.9cm]{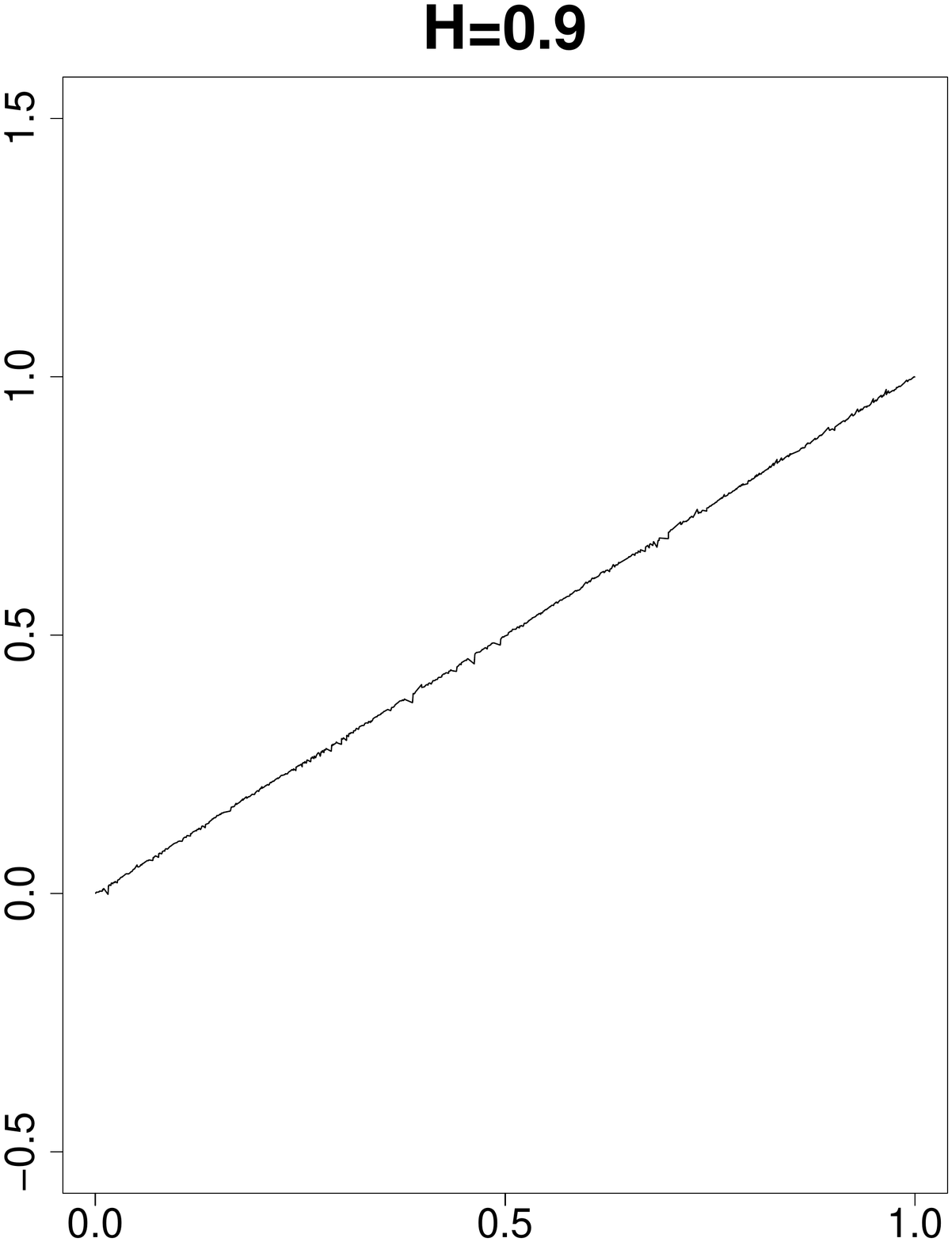}
\end{subfigure}
	\caption{Regression data $Y_\tau$ for $N=1000$ and $H\in\{0.1,0.5,0.9\}$.}\label{data}
\end{figure}

We consider three possible values for $N$: $N=100$, $N=1000$, and  $N=5000$. In order to assess the finite sample properties of the estimator, we conduct a Monte Carlo simulation study with a number of replications $M=1000$.

The absolute mean bias over the $M=1000$ {replications}, defined for the $k$:th replication by
$$\left|\frac{1}{M}\sum_{k=1}^M\hat{a}_{N(1)}^{(k)}-a\right|,$$
for different values of $N$ and $H$ is presented in Table \ref{table.CompM}. As expected, the bias decreases as $N$ increases for all values of $H$.

The $L_1$ risk of the estimator over the $M=1000$ {replications}, defined for the $k$:th replication by
$$\frac{1}{M}\sum_{k=1}^M\left|\hat{a}_{N(1)}^{(k)}-a\right|,$$
is presented in Table \ref{table.CompMbis}. The estimated variance of {$N(1)(\hat{a}_{N(1)}-a)$}, computed in a standard way by
$$\frac{1}{M}\sum_{k=1}^M \left( N(1)^{(k)}    \left[   \hat{a}_{N(1)}^{(k)}-a \right]  \right) ^2,$$
where $N(1)^{(k)}$ is the number of observations needed to reach the time $1$ in the $k$:th replication, is presented in Table \ref{tableVar}. It can be seen that, as expected, the estimated variance is close to the theoretical variance $9/(2H+2)$.

\begin{table}[htp!]
	\centering
\caption{\label{table.CompM} Absolute mean bias of the estimator for different values of $N$ and $H$. }
\vspace{0.3cm}
\fbox{
\begin{tabular}{cccc}
 & $H=0.1$ & $H=0.5$ & $H=0.9$\\
\hline
$N=100$ & 6.92 $\cdot 10^{-4}$ & 7.44 $\cdot 10^{-4}$& 1.69 $\cdot 10^{-4}$\\
$N=1000$ & 6.15 $\cdot 10^{-5}$ & 4.70 $\cdot 10^{-6}$ & 8.31 $\cdot 10^{-6}$\\
$N=5000$ & 9.66 $\cdot 10^{-6}$ & 3.97 $\cdot 10^{-6}$ & 3.11 $\cdot 10^{-6}$\\
\end{tabular}}
\end{table}
\begin{table}[htp!]
 	\centering
 \caption{\label{table.CompMbis} $L_1$ risk of the estimator for different values of $N$ and $H$. }
 \vspace{0.3cm}
 \fbox{
 \begin{tabular}{cccc}
  & $H=0.1$ & $H=0.5$  & $H=0.9$\\
 \hline
 $N=100$ & 1.65 $\cdot 10^{-2}$ &1.47 $\cdot 10^{-2}$& 1.25 $\cdot 10^{-2}$\\
 $N=1000$ & 1.68 $\cdot 10^{-3}$ & 1.39 $\cdot 10^{-3}$ & 1.25 $\cdot 10^{-3}$\\
 $N=5000$& 3.22 $\cdot 10^{-4}$ & 2.85 $\cdot 10^{-4}$& 2.39 $\cdot 10^{-4}$\\
 \end{tabular}}
 \end{table}
\begin{table}[htp!]
		\centering
\caption{\label{tableVar} Asymptotic variance of {$N(1)(\hat{a}_{N(1)}-a)$} for different values of $N$ and $H$ over the $M=1000$ replications.}
\vspace{0.3cm}
\fbox{
\begin{tabular}{cccc}
Theoretical & $H=0.1$ & $H=0.5$ & $H=0.9$\\
  variance & 4.09& 3 & 2.37\\
\hline
Estimated variance & $H=0.1$& $H=0.5$ & $H=0.9$\\
$N=100$ & 4.221  & 3.250 & 2.352\\
$N=1000$ & 4.359  & 3.027 & 2.372\\
$N=5000$ & 4.09 &  3.13 &  2.19\\
\end{tabular}}
\end{table}
Figure~\ref{hist01}
shows the plots of the histograms for $N(1) \left( \hat{a}_{N(1)} -a \right)$ with values $N\in\{100,1000,5000\}$ and $H=0.1$, jointly with the density of the normal distribution with mean $0$ and variance $9/(2H+2)$ corresponding to the asymptotic distribution, cf. Theorem~\ref{thm:N1}. By plotting histograms for other values of $H$, one obtains similar figures and the asymptotic normality with mean $0$ and variance $9/(2H+2)$.
\begin{figure}[htp!]
		\centering
	\begin{subfigure}
		\centering
		\includegraphics[width=3.9cm]{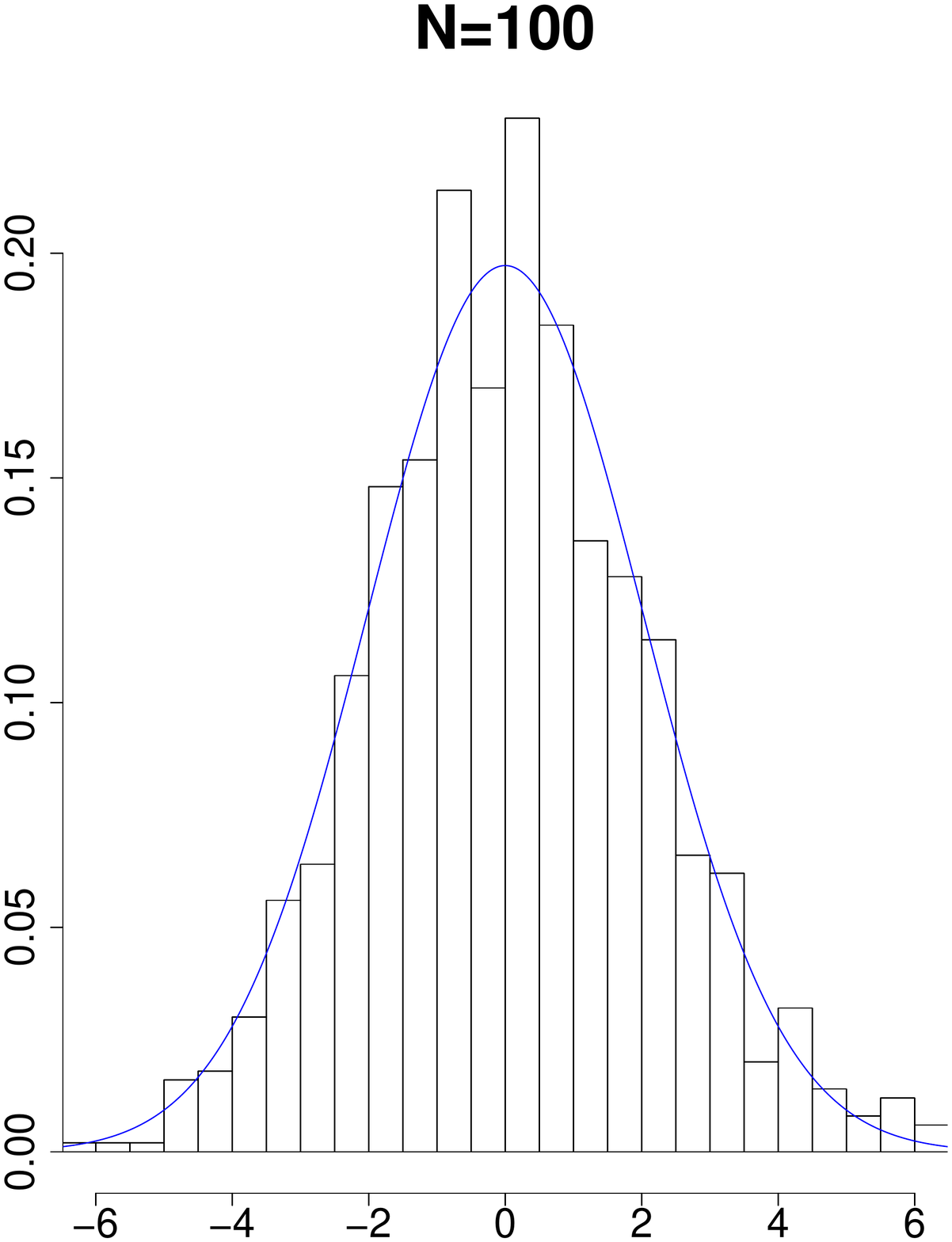}
	\end{subfigure}
	\begin{subfigure}
		\centering
		\includegraphics[width=3.9cm]{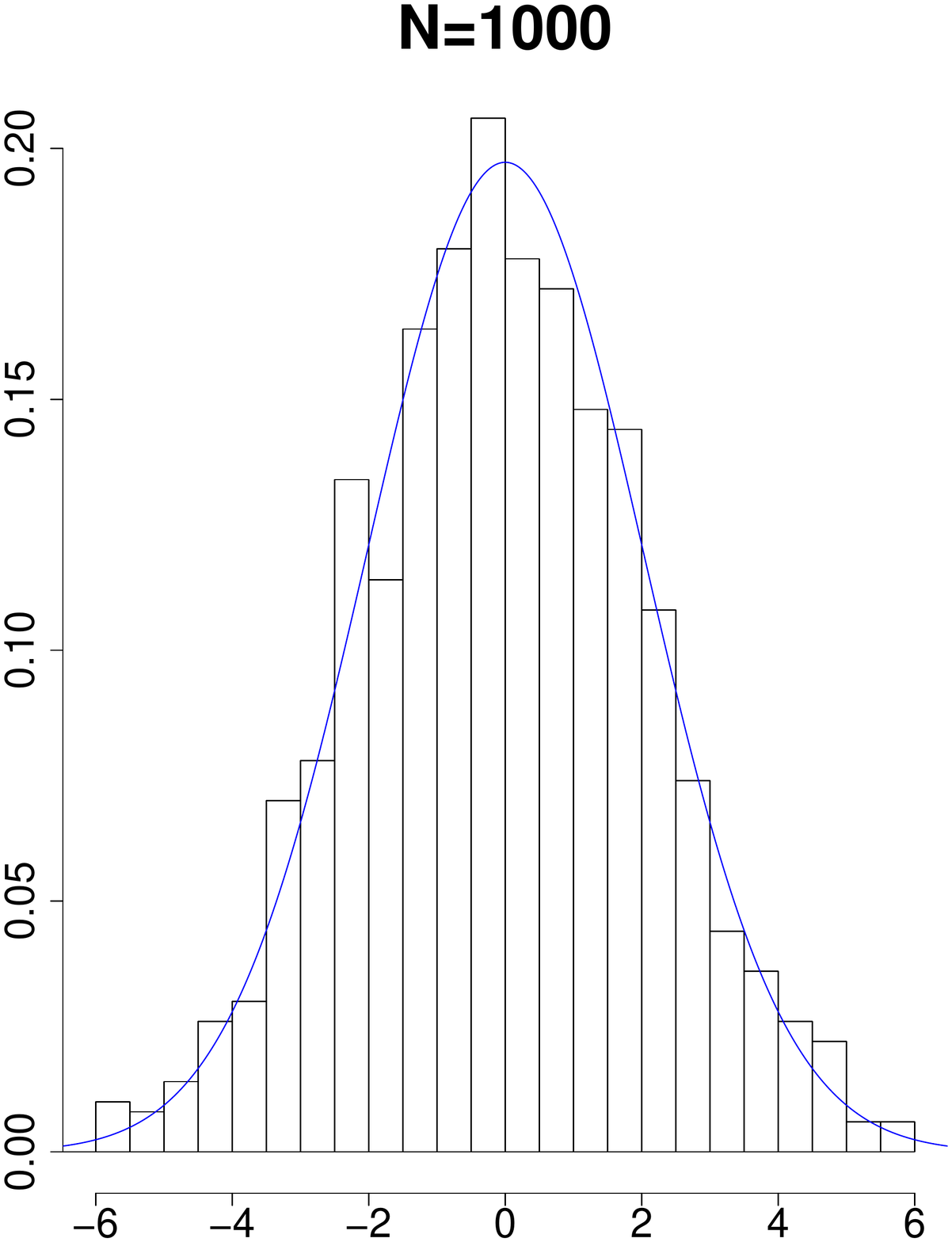}
	\end{subfigure}
	\begin{subfigure}
	\centering
	\includegraphics[width=3.9cm]{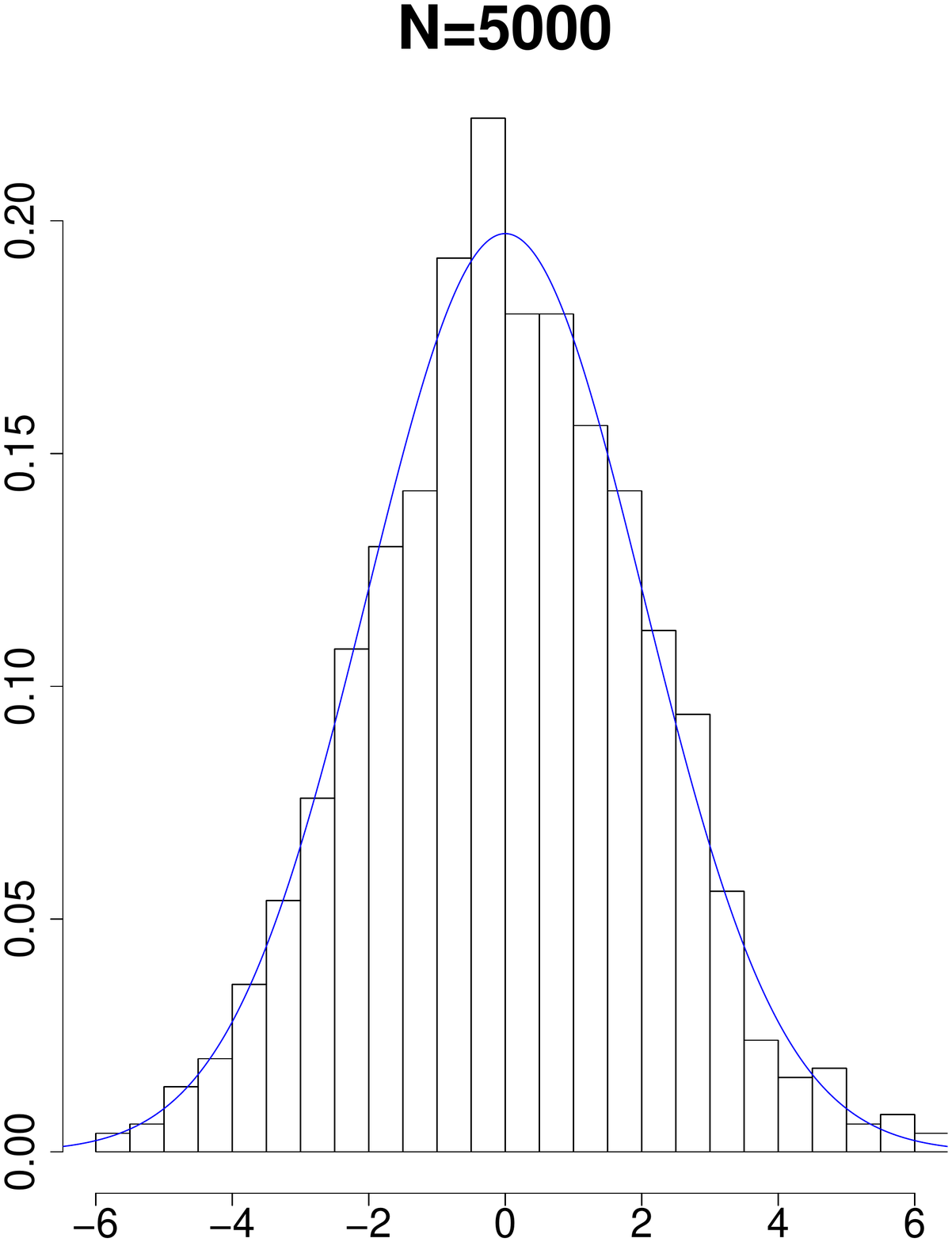}
\end{subfigure}
	\caption{Histograms for $N(1) \left( \hat{a}_{N(1)} -a \right) $ for  $H=0.1$ and $N\in\{100,1000,5000\}$.}\label{hist01}
\end{figure}
Finally, Figure~\ref{figN1} shows the ratio between the numbers $N(1)$ and $N$. It is clearly visible in Figure \ref{figN1} that as $N$ increases, the ratio $N(1)/N$ is more concentrated around $1$. This is in line with Theorems \ref{thm:N1} and \ref{thm:N} (see also Lemma \ref{lma:N1-N-ratio} below). Again we have only plotted the case $H=0.1$, while similar phenomena can be observed for other values of $H$ as well.
\begin{figure}[htp!]
		\centering
	\begin{subfigure}
		\centering
		\includegraphics[width=3.9cm]{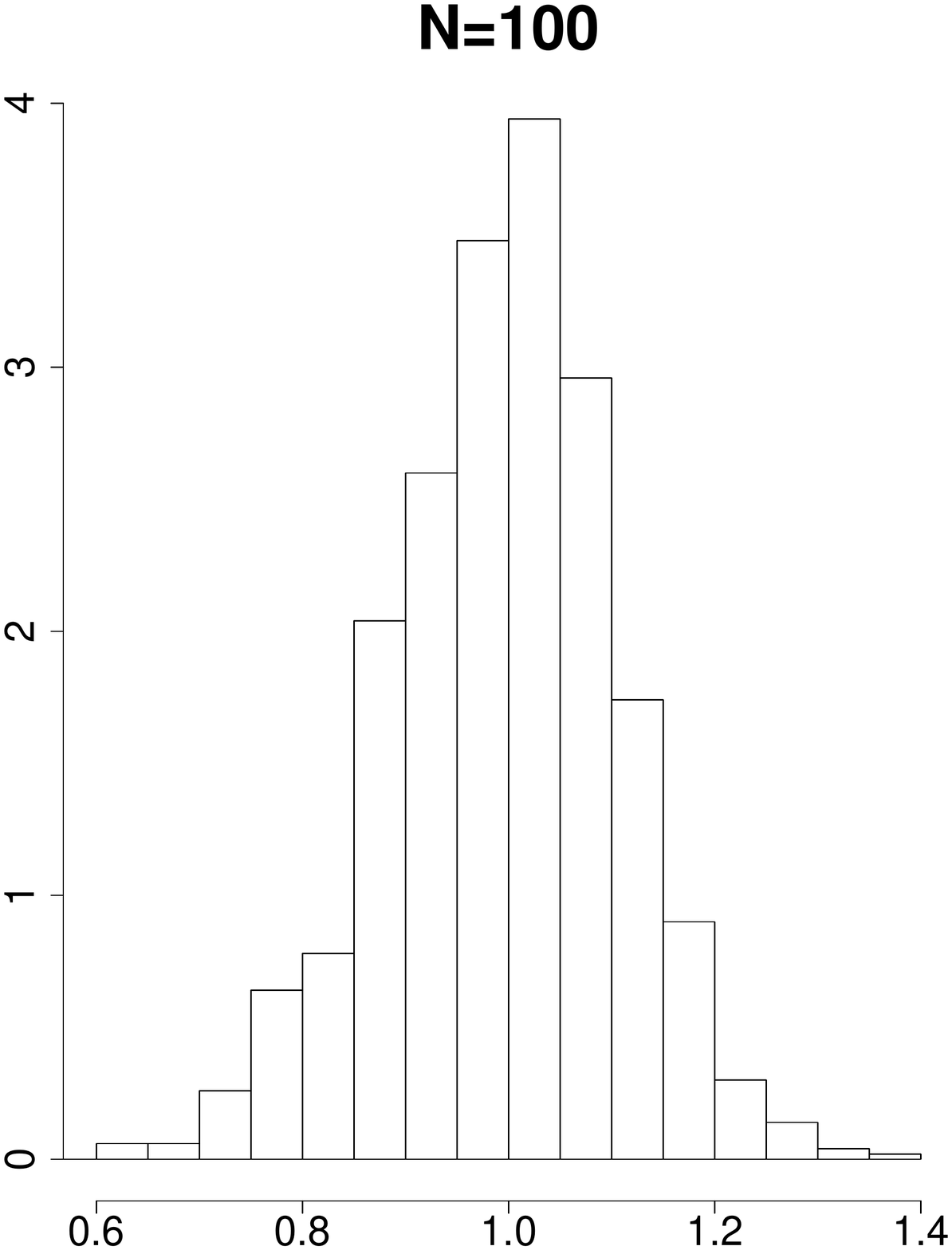}
	\end{subfigure}
	\begin{subfigure}
		\centering
		\includegraphics[width=3.9cm]{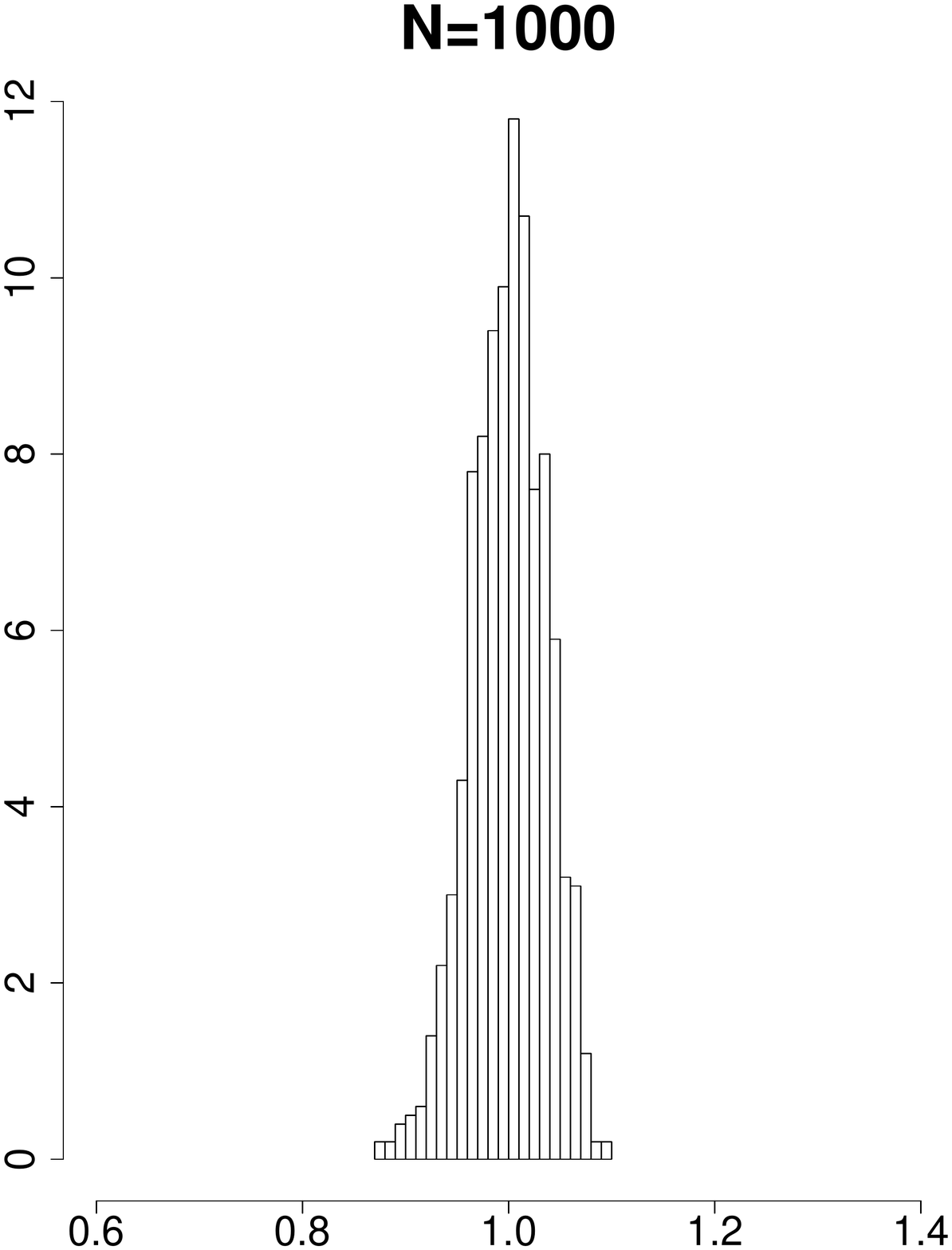}
	\end{subfigure}
	\begin{subfigure}
		\centering
		\includegraphics[width=3.9cm]{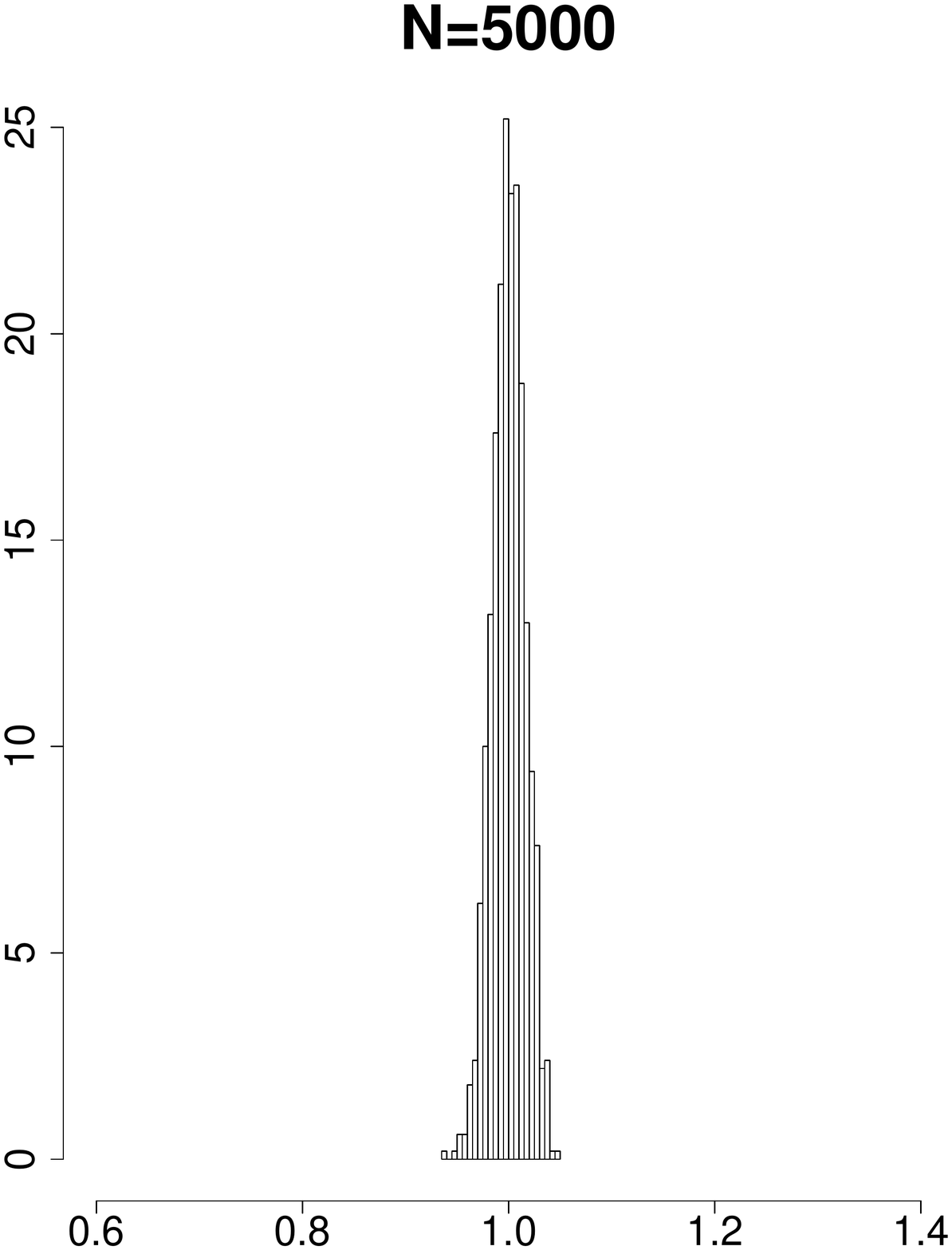}
	\end{subfigure}
	\caption{Histograms for $N(1) /N $  for  $H=0.1$ and $N\in\{100,1000,5000\}$.}\label{figN1}
\end{figure}

\section{Proofs}\label{proofs}
This section is devoted to the proofs of our main results, Theorem \ref{thm:N1} and Theorem \ref{thm:N}. We begin by presenting some auxiliary results, while the proofs of the main results are presented in Subsection \ref{subsec:main-proof}. Throughout the proofs we denote by $C$ a generic unimportant constant which may vary from line to line.
\subsection{Auxiliary results}
\begin{lem}
\label{lma:maximum}
Suppose that the sequence  $\{ t_{j} \}_{j \in \mathbb{N}}$   satisfies hypothesis  $(H_2)$. Then we have $ \displaystyle \max_{1\leq j\leq N} t_{j} \to 0$ almost surely as $N \to \infty$.
\end{lem}
\begin{proof}
Let $\delta>0$ be fixed and denote
$$
S_N(\delta) = \sum_{j=1}^N \textbf{1}_{t_{j}\geq \delta}.
$$
Then we have 
\begin{equation*}
\begin{split}
{\mathbb{P}\left(\max_{1\leq j\leq N} t_{j}>\delta\right) }
&\leq   \mathbb{P}\left(S_N(\delta) \geq 1\right)\\
&\leq  \E (S_N(\delta))  =  \sum_{j=1}^N \mathbb{P}\left(t_{j}\geq \delta\right) \\
&\leq  \delta^{-2-r}\sum_{j=1}^N \E \left(t_{j}^{2+r}\right).
\end{split}
\end{equation*}
Hence, by $(H_2)$, we obtain
$$
\mathbb{P}\left(\max_{1\leq j\leq N} t_{j}>\delta\right) \leq \delta^{-2-r}\tilde{C} N^{-1-r}
$$
and consequently,
$$
\sum_{N=1}^\infty\mathbb{P}\left(\max_{1\leq j\leq N} t_{j}>\delta\right) < \infty.
$$
The claim follows from Borel--Cantelli lemma.
\end{proof}
Note that if the sequence  $\{t_{j} \}_{j \in \mathbb{N}}$ is  NSD, then random series based on $t_{j}$ behaves essentially {as random series based on the  independent $t^*_{j}$'s that satisfy Definition~\ref{def:hu}}. In particular, taking $X_j=t_{j}-\mathbb{E}(t_{j})$, we have the Rosenthal inequality (see \cite{S2}) for any $p>2$
{\begin{equation}\label{ros}
\mathbb{E}\left(\max_{1\le k\le n} |\sum_{j=1}^k  X_j|^p\right)\le 2\left(\frac{15 p}{\log p}\right)^p \left[ \sum_{j=1}^n \mathbb{E}|X_j|^p+\left(\sum_{j=1}^n \mathbb{E}(X_j^2)  \right)^{p/2}\right],
\end{equation}}
provided that $p>2$ and $\mathbb{E}(t_{j}^p)<\infty$. Moreover, for $p=2$ we have
{\begin{equation}\label{var}
\mathbb{E}\left( \left(\sum_{j=1}^n  X_j\right)^2\right)\le \sum_{j=1}^n \mathbb{E}(X_j^2)  .
\end{equation}}
This follows directly from the fact that
{$$
(x_1,\ldots,x_{n}) \mapsto \left[\sum_{j=1}^{n} \left[x_j - \E t_{j}\right]\right]^2
$$}
is superadditive, which can be seen by taking partial derivatives and using Lemma \ref{lem:kemp}. We also use the following proposition in the proofs of our main results. While the result could be obtained by verifying several conditions of Theorem 3.2 in \cite{S2}, here we present a more concise proof for the reader's convenience.
\begin{proposition}
\label{prop:NSD-as-limits}
Suppose that the sequence  $\{ t_{j} \}_{j \in \mathbb{N}}$ {satisfies} $(H_1)$-$(H_3)$ and that $g:\mathbb{N} \mapsto \mathbb{N}$ is a function such that, as $N\to \infty$, we have
$$
\frac{g(N)}{N} \to L
$$
for some $L>0$.
Then
$
\sum_{j=1}^{g(N)} \left[t_{j} - \E t_{j}\right]
$
converges almost surely to 0 as $N\to \infty$ and consequently,
$$
\sum_{j=1}^{g(N)} t_{j} \to L
$$
almost surely.
\end{proposition}
\begin{proof}
We have
$$\E((t_{j}-\E t_{j})^2)\le \frac{\tilde{C}^{2/(2+r)}}{N^2}$$
and
$$\E((t_{j}-\E t_{j})^{2+r})\le 2^{r-1}\left( \E(t_{j})^{2+r}+ (\E(t_{j}))^{2+r} \right)\le C N^{-2-r}$$
Using \eqref{ros} with $p=2+r$, we have
\begin{equation}
\E\left( \left|\sum_{j=1}^{g(N)}  t_{j}-\E t_{j}\right|^{2+r}\right)\le C\left(g(N)N^{-2-r}+\left(g(N)N^{-2}  \right)^{1+r/2}   \right).
\end{equation}
Using $g(N)/N \to L$ we obtain that, for all $\varepsilon>0$,
\begin{equation}\label{eqconv}\sum_{N=1}^{\infty} \mathbb{P}\left(  \left|\sum_{j=1}^{g(N)}  t_{j}-\E t_{j}\right|\ge \varepsilon \right)\le \frac{1}{\varepsilon^{2+r}}\sum_{N=1}^{\infty}  \left(\frac{g(N)}{N}\right)N^{-r-1}+ \left(\frac{g(N)}{N}\right)^{1+r/2}N^{-1-r/2} <\infty.\end{equation}
This together with Borel-Cantelli lemma implies that $\sum_{j=1}^{g(N)}  (t_{j}-\E t_{j})$ converges to $0$ almost surely. Finally, $\sum_{j=1}^{g(N)}  t_{j} \to L$ follows from the fact $\sum_{j=1}^{g(N)}\E t_{j}=g(N)/N$.

\end{proof}
With Proposition \ref{prop:NSD-as-limits}, we can deduce the following two lemmas.
\begin{lem}
\label{lma:tauN-limit}
Suppose that the sequence  $\{ t_{j} \}_{j \in \mathbb{N}}$ satisfy {hypotheses} $(H_1)$-$(H_3)$. Then, as $N\to \infty$, $\tau_{N} \to 1$ almost surely.
\end{lem}
\begin{proof}
By the very definition we have
$$
\tau_{N} = \sum_{j=1}^N t_{j}
$$
so it remains to apply Proposition \ref{prop:NSD-as-limits} with $g(N) = N$.
\end{proof}
\begin{lem}
\label{lma:N1-N-ratio}
Suppose that the sequence  $\{ t_{j} \}_{j \in \mathbb{N}}$ satisfy $(H_1)$-$(H_3)$ and let $N(1)$ be given by \eqref{eq:N(1)}. Then, as $N\to \infty$, we have
$
\frac{N(1)}{N} \to 1
$
and $|\tau_{N(1)}-\tau_{N}| \to 0$ almost surely.
\end{lem}
\begin{proof}
For a fixed $\epsilon>0$, define sets $A_+(\epsilon)$ and $A_-(\epsilon)$ by
$$
A_+(\epsilon) = \left\{\omega : \limsup_{N\to \infty} \frac{N(1)}{N} \geq 1+\epsilon\right\}
$$
and
$$
A_-(\epsilon) = \left\{\omega : \liminf_{N\to \infty} \frac{N(1)}{N} \leq 1-\epsilon\right\}.
$$
If $\omega \in A_+(\epsilon)$, then we can find arbitrary large numbers $N$ such that $N(1) \geq (1+\epsilon)N$ leading to
$$
\tau_{N(1)} = \sum_{j=1}^{N(1)} t_{j} \geq \sum_{j=1}^{\lfloor (1+\epsilon)N\rfloor}t_{j}.
$$
On the other hand, we obtain by Proposition \ref{prop:NSD-as-limits} as $N \to \infty$ that
$$
\sum_{j=1}^{\lfloor (1+\epsilon)N\rfloor }t_{j} \to 1+\epsilon.
$$
Since $\tau_{N(1)} \leq 1$ by the very definition of $N(1) = \max_{i \in \mathbb{N}} \{\tau_i :  \,\, \tau_i \leq 1\}$, it follows that $\mathbb{P}(A_+(\epsilon)) = 0$. Similarly, for $\omega \in A_-(\epsilon)$ we can find arbitrary large $N$ such that $N(1)\leq (1-\epsilon)N$ and get
$$
\tau_{N(1)+1} \leq \sum_{j=1}^{\lfloor (1-\epsilon)N\rfloor +1}t_{j}.
$$
We observe again by Proposition \ref{prop:NSD-as-limits} as $N \to \infty$,
$$
\sum_{j=1}^{\lfloor (1-\epsilon)N\rfloor +1}t_{j} \to 1-\epsilon.
$$
On the other hand, $\tau_{N(1)+1} \geq 1$ and thus $\mathbb{P}(A_-(\epsilon)) = 0$ as well. It follows that for any $\epsilon>0$ we have, almost surely,
$$
1-\epsilon < \liminf_{N\to \infty} \frac{N(1)}{N} \leq \limsup_{N\to \infty} \frac{N(1)}{N} < 1+\epsilon.
$$
Since $\epsilon>0$ is arbitrary, it follows that $\lim_{N\to \infty} \frac{N(1)}{N} = 1$ almost surely. For the second claim, by the very definition of $N(1)$ we get

$$
0\leq 1-\tau_{N(1)} \leq   \tau_{N(1) +1 } - \tau_{N(1)}  = t_{N(1)+1}
$$
leading to
$$
0\leq 1-\tau_{N(1)} \leq t_{N(1)+1}.
$$
Now since $\frac{N(1)+1}{N}\to 1$ almost surely, it follows that there exists $N_0 = N_0(\omega)$ such that $N(1) +1  \leq 2N$ for all $N\geq N_0$. This gives us
$$
t_{N(1)+1} \leq \max_{1\leq j \leq 2N} t_{j}
$$
which converges to zero by the arguments in the proof of Lemma \ref{lma:maximum}. Thus
$$
|1- \tau_{N(1)}| \to 0
$$
from where $|\tau_{N}-\tau_{N(1)}|\to 0$ follows by Lemma \ref{lma:tauN-limit}. This completes the proof.
\end{proof}
We end this subsection with the following proposition on the denominator in \eqref{eq:estimator-difference}.
\begin{proposition}
\label{prop:D-part}
Suppose $\{t_{j}\}_{j\in\mathbb{N}}$ satisfies $(H_1)$ to $(H_3)$. Then, almost surely as $N\to \infty$,
$$
\frac{1}{N} \sum_{k=1}^N \tau_{k}^2 \to \frac13.
$$
\end{proposition}
\begin{proof}
In order to prove the claim, we first observe that it suffices to show that
\begin{equation}
\label{eq:DN-needed}
\frac{1}{N} \sum_{k=1}^N \left(\tau_{k}-\E \tau_{k}\right)^2 \to 0,
\end{equation}
almost surely.
Indeed, we may write
\begin{equation*}
\begin{split}
\frac{1}{N} \sum_{k=1}^N \tau_{k}^2 &= \frac{1}{N} \sum_{k=1}^N \left(\tau_{k}-\E \tau_{k}\right)^2 \\
&+ 2\frac{1}{N} \sum_{k=1}^N \left(\tau_k-\E \tau_{k}\right)\E \tau_{k} +\frac{1}{N} \sum_{k=1}^N \left(\E \tau_{k}\right)^2.
\end{split}
\end{equation*}
Here
\begin{equation*}
\frac{1}{N} \sum_{k=1}^N \left(\E \tau_{k}\right)^2  =\frac{1}{N} \sum_{k=1}^N \left(\frac{k}{N}\right)^2=\frac{N(N+1)(2N+1)}{6N^3} \to \frac13
\end{equation*}
as $N$ tends to infinity and, by Cauchy-Schwartz inequality,
\begin{equation*}
\begin{split}
&\left|\frac{1}{N} \sum_{k=1}^N \left(\tau_{k}-\E \tau_{k}\right)\E \tau_{k}\right| \\
&\leq \sqrt{\frac{1}{N} \sum_{k=1}^N \left(\tau_{k}-\E \tau_{k}\right)^2}\sqrt{\frac{1}{N} \sum_{k=1}^N \left(\E \tau_{k}\right)^2}
\end{split}
\end{equation*}
that converges to zero once \eqref{eq:DN-needed} is proved. Denote $\tilde{\tau}_{j} = \tau_{j} - \E \tau_{j}$ and $\tilde{t}_{j} = t_{j} - \E t_{j}$, and let $\epsilon>0$ be fixed. Markov inequality gives us
\begin{equation}
\label{eq:DN-probability}
\mathbb{P}\left(\frac{1}{N}\sum_{k=1}^N \tilde{\tau}_{k}^2 > \epsilon\right) \leq \frac{\E\left(\sum_{k=1}^N \tilde{\tau}_{k}^2\right)^{\frac{2+r}{2}}}{(\epsilon N)^{\frac{2+r}{2}}}.
\end{equation}
Furthermore, by Minkowski integral inequality we obtain
\begin{equation}
\label{eq:DN-Minkowski}
\E\left(\sum_{k=1}^N \tilde{\tau}_{k}^2\right)^{\frac{2+r}{2}}
\leq \left[\sum_{k=1}^N \left(\E |\tilde{\tau}_{k}|^{2+r}\right)^{\frac{2}{2+r}}\right]^{\frac{2+r}{2}}.
\end{equation}
By applying $(H_2)$ and the Rosenthal inequality \eqref{ros}, we can deduce
\begin{equation*}
\begin{split}
\E |\tilde{\tau}_{k}|^{2+r}
&\leq C\sum_{j=1}^N \E|\tilde{t}_{j}|^{2+r} + C\left(\sum_{j=1}^N \E \tilde{t}_{j}^2\right)^{\frac{2+r}{2}} \\
&\leq C N^{-1-r} + C N^{-\frac{2+r}{2}}\leq C N^{-1-r/2}.
\end{split}
\end{equation*}
Plugging into \eqref{eq:DN-Minkowski} leads to
{$$
\E\left(\sum_{k=1}^N \tilde{\tau}_{k}^2\right)^{\frac{2+r}{2}} \leq C N^{-1-r/2} N^{\frac{2+r}{2}}\le C.
$$}
In view of \eqref{eq:DN-probability} we get
{$$
\sum_{N=1}^\infty \mathbb{P}\left(\frac{1}{N}\sum_{k=1}^N \tilde{\tau}_{k}^2 > \epsilon\right) \le C\sum_{N=1}^\infty  (\epsilon N)^{-1-r/2}< \infty,
$$}
and hence \eqref{eq:DN-needed} follows from Borel-Cantelli lemma. This completes the proof.
\end{proof}

\subsection{Proofs of Theorem \ref{thm:N1} and Theorem \ref{thm:N}}
\label{subsec:main-proof}

\begin{proof}[Proof of Theorem \ref{thm:N1}]

First we will prove that
\begin{equation}\label{eq:denom}
\frac{1}{N(1)} \sum_{j=1}^{N(1)} \tau_{j}^2 \to \frac{1}{3}.
\end{equation}
With the convention $\sum_{j=1}^0 a_j = 0$ we can write
$$
\sum_{j=1}^{N(1)} \tau_{j}^2 = \sum_{j=1}^{N} \tau_{j}^2 + \sum_{j=N+1}^{N(1)} \tau_{j}^2
$$
for $N(1)\geq N$ and
$$
\sum_{j=1}^{N(1)} \tau_{j}^2 = \sum_{j=1}^{N} \tau_{j}^2 - \sum_{j=N(1)+1}^{N} \tau_{j}^2
$$
for $N(1)<N$. Since $\tau_{N(1)}\leq 1$, here
$$
\frac{1}{N(1)}\sum_{j=N+1}^{N(1)} \tau_{j}^2 \leq {\frac{|N-N(1)|}{N(1)}} \to 0
$$
by Lemma \ref{lma:N1-N-ratio}.
Similarly,
$$
\frac{1}{N(1)}\sum_{j=N(1)+1}^{N} \tau_{j}^2 \leq \tau^2_{N}{\frac{|N-N(1)|}{N(1)}} \to 0
$$
by Lemmas~\ref{lma:tauN-limit} and \ref{lma:N1-N-ratio}. Finally, we may apply Lemma \ref{lma:N1-N-ratio} and Proposition \ref{prop:D-part} to obtain
\begin{equation}\label{denom}
\frac{1}{N(1)} \sum_{j=1}^N \tau_{j}^2 = \frac{N}{N(1)} \frac{1}{N}\sum_{j=1}^N \tau_{j}^2 \to \frac13.
\end{equation}
Second, we will prove that
\begin{equation}\label{eq:num}
\sum_{j=1}^{N(1)} \tau_{j} \left(W_{\tau_{j}}- W_{\tau_{j-1}}\right) \to \int_0^1 (W_1 - W_t) dt.
\end{equation}
Using summation by parts we obtain
$$
\sum_{j=1}^{N(1)} \tau_{j} \left(W_{\tau_{j}}- W_{\tau_{j-1}}\right) = \tau_{N(1)} W_{\tau_{N(1)}} - \sum_{j=1}^{N(1)} W_{\tau_{j-1}}(\tau_{j}-\tau_{j-1}).
$$
Since $W$ is almost surely continuous, Lemmas~\ref{lma:tauN-limit} and \ref{lma:N1-N-ratio} implies that $\tau_{N(1)} W_{\tau_{N(1)}} \to W_1 = \int_0^1 W_1 dt$. We have
$$
\sum_{j=1}^{N(1)} W_{\tau_{j-1}}(\tau_{j}-\tau_{j-1})
= \sum_{j=1}^{N(1)+1} W_{s_{j-1}}(s_{j}-s_{j-1}) - W_{\tau_{N(1)}}(1-\tau_{N(1)}).
$$
Since $W$ is bounded almost surely as a continuous function,  $W_{\tau_{N(1)}}(1-\tau_{N(1)})$ converges almost surely to $0$.

Let us introduce the sequence $s_{j}$ by $s_{j} = \tau_{j}$ for $j=1,\ldots,N(1)$ and $s_{N(1)+1} = 1$.  This sequence satisfies
$$
\max_{1\leq j \leq N(1)+1}(s_{j}-s_{j-1}) \leq |1-\tau_{N(1)}| + \max_{1\leq j \leq N(1)}t_{j}
$$

Now, since {$\frac{N(1)}{N} \to 1$} we obtain that, in particular, for every $\omega$ there exists $N_0(\omega)$ such that for all $N\ge N_0(w)$ we have $\frac{N}{2}\leq N(1) \leq 2N$. Following the same lines as Lemma \ref{lma:maximum} we deduce that
$$
\max_{1\leq j \leq 2N}t_{j} \to 0
$$
which implies that
$$
\max_{1\leq j \leq N(1)+1}(s_{j}-s_{j-1}) \to 0.
$$
Finally, by continuity of $W$ the Riemann integral $\int_0^1 W_t dt$ exists, and we have
$$
\sum_{j=1}^{N(1)+1} W_{s_{j-1}}(s_{j}-s_{j-1}) \to \int_0^1 W_t dt
$$
almost surely.
Hence we obtain \eqref{eq:num} and the theorem is a consequence of \eqref{eq:denom} and \eqref{eq:num}.
\\
\end{proof}
\begin{proof}[Proof of Theorem~\ref{thm:N}]

By Proposition \ref{prop:D-part} it suffices to show
\begin{equation}\label{eq_num2}
\sum_{j=1}^{N} \tau_{j} \left(W_{\tau_{j}}- W_{\tau_{j-1}}\right) \to \int_0^1 (W_1 - W_t)  \,dt
\end{equation}
almost surely as $N\to \infty$.
Two cases have to be considered: $\tau_{N}<1$ and $\tau_{N}\ge 1$. When $\tau_{N}<1$, we can mimick the proof of Theorem~\ref{thm:N1} in order to derive \eqref{eq_num2}.

When $\tau_{N}\geq 1$, we write
$$
\sum_{j=1}^N W_{\tau_{j-1}}(\tau_{j}-\tau_{j-1})
= \sum_{j=1}^{N(1)} W_{\tau_{j-1}}(\tau_{j}-\tau_{j-1})  + \sum_{j=N(1)+1}^N W_{\tau_{j-1}}(\tau_{j}-\tau_{j-1}).
$$
We prove in the proof of Theorem~\ref{thm:N1} that $\sum_{j=1}^{N(1)} W_{\tau_{j-1}}(\tau_{j}-\tau_{j-1})$ converges almost surely to $\int_0^1 (W_1 - W_t) dt$. For the second sum we obtain
$$
\left|\sum_{j=N(1)+1}^N W_{\tau_{j-1}}(\tau_{j}-\tau_{j-1})\right| \leq \sup_{0\leq s \leq \tau_{N}}|W_s||\tau_{N}-\tau_{N(1)}|
$$
which converges to zero by boundedness of $W$ and Lemma \ref{lma:N1-N-ratio}. This allows us to deduce \eqref{eq_num2}.
\end{proof}

\section*{Acknowledgements}
Karine Bertin and Soledad Torres have been supported by FONDECYT grants 1171335 and 1190801 and Mathamsud 20MATH05.

\bibliographystyle{plainnat}
\bibliography{biblio}   

\end{document}